\documentclass{amsart}
\usepackage{amssymb}
\usepackage{amsthm}
\usepackage[curve,matrix,arrow]{xy}
\theoremstyle{plain}\newtheorem{Theorem}{Theorem}[section]
\theoremstyle{plain}
\theoremstyle{plain}\newtheorem{Corollary}[Theorem]{Corollary}
\theoremstyle{plain}\newtheorem{Lemma}[Theorem]{Lemma}
\theoremstyle{plain}\newtheorem{Proposition}[Theorem]{Proposition}
\theoremstyle{plain}
\theoremstyle{definition}
\theoremstyle{definition}
\theoremstyle{definition}
\theoremstyle{definition}
\theoremstyle{definition}\newtheorem{Remark}[Theorem]{Remark}
\theoremstyle{definition}

\newcommand{\HH}{H\!H}

\def\Der{\mathrm{Der}}
\def\dim{\mathrm{dim}}           
\def\End{\mathrm{End}}           

\def\Ext{\mathrm{Ext}}

\def\ker{\mathrm{ker}}           
             
\def\IDer{\mathrm{IDer}}
\def\Im{\mathrm{Im}}

\def\soc{\mathrm{soc}}

\title{Block algebras with $\HH^1$ a simple Lie algebra} 
\author{Markus Linckelmann and Lleonard Rubio y Degrassi} 
 \date{}


\begin{document}

\begin{abstract}
We show that if $B$ is a block of a finite group algebra $kG$ over an 
algebraically closed field $k$ of prime characteristic $p$ such that
$\HH^1(B)$ is a simple Lie algebra and such that $B$ has a unique
isomorphism class of simple modules, then $B$ is nilpotent with an
elementary abelian defect group $P$ of order at least $3$, and 
$\HH^1(B)$ is in that case isomorphic to the Jacobson-Witt algebra 
$\HH^1(kP)$. In particular, no other simple modular Lie algebras arise 
as $\HH^1(B)$ of a block $B$ with a single isomorphism class of simple 
modules.
\end{abstract}

\maketitle

\section{Introduction}

Let $p$ be a prime and $k$ an algebraically closed field of 
characteristic $p$. The purpose of this note is to illustrate close
connections between the Lie algebra structure of $\HH^1(B)$ and
the structure of $B$, where $B$ is a block of a finite group algebra
$kG$. The main motivation for this is the fact that the Lie algebra
structre of $\HH^1(B)$ is invariant under stable equivalences of 
Morita type (cf. \cite[Theorem 10.7]{KLZ}). We consider two extreme 
cases for a block $B$ with a single isomorphism class of simple 
modules. The first result describes the circumstances under which
$HH^1(B)$ a simple Lie algebra. 

\begin{Theorem} \label{onesimpleLie}
Let $G$ be a finite group and let $B$ be a block algebra of $kG$ having
a unique isomorphism class of simple modules. Then $\HH^1(B)$ is a 
simple Lie algebra if and only if $B$ is nilpotent with an elementary 
abelian defect group $P$ of order at least $3$. In that case, we have 
a Lie algebra isomorphism $\HH^1(B)\cong$ $\HH^1(kP)$.
\end{Theorem}

Theorem \ref{onesimpleLie} implies in particular that no simple modular 
Lie algebras other than the Jacobson-Witt algebras occur as $\HH^1(B)$ 
of some block algebra of a finite group with the property that $B$ has 
a single isomorphism class of simple modules. See \cite{StradeI}, 
\cite{StradeII}  for details and further references on the 
classification of simple Lie algebras in positive characteristic. We do 
not know whether the hypothesis on $B$ to have a single isomorphism 
class of simple modules is necessary in Theorem \ref{onesimpleLie}. For 
the sake of completeness, the second result rules out the case of the 
trivial one-dimensional Lie algebra for blocks with one isomorphism 
class of simple modules.

\begin{Theorem} \label{onetrivialLie}
Let $G$ be a finite group and let $B$ be a block algebra of $kG$ having
a nontrivial defect group and a unique isomorphism class of simple 
modules. Then $\dim_k(\HH^1(B))\geq$ $2$.
\end{Theorem}

The hypothesis that $B$ has a single isomorphism class of simple
modules is necessary in Theorem \ref{onetrivialLie}; for instance,
if $P$ is cyclic of order $p\geq$ $3$ and if $E$ is the cyclic 
automorphism group of order $p-1$ of $P$, then $\HH^1(k(P\rtimes E))$ 
has dimension one. 

\section{Quoted results}

We collect in this section results needed for the proof of
Theorem \ref{onesimpleLie}.

\begin{Theorem}[Okuyama and Tsushima \cite{OkTs}] 
\label{abeliandefectnilpotent}
Let $G$ be a finite group and $B$ a block algebra of $kG$. Then $B$ is 
a nilpotent block with an abelian defect group if and only if $J(B)=$ 
$J(Z(B))B$.
\end{Theorem}

Let $A$ be a finite-dimensional (associative and unital) $k$-algebra. 
A derivation on $A$ is a $k$-linear map $f : A\to$ $A$ satisfying 
$f(ab)=$ $f(a)b+af(b)$ for all $a$, $b\in$ $A$. The set $\Der(A)$ of 
derivations on $A$ is a Lie subalgebra of $\End_k(A)$, with respect to 
the Lie bracket $[f,g]=$ $f\circ g- g\circ f$, for any $f$, $g\in$ 
$\End_k(A)$. For $c\in$ $A$, the map sending $a\in$ $A$ to the additive 
commutator $[c,a]=$ $ca-ac$ is a derivation on $A$; any derivation 
arising this way is called an {\it inner derivation on} $A$. The set 
$\IDer(A)$ of inner derivations is a Lie ideal in $\Der(A)$, and we 
have a canonical identification $\HH^1(A)\cong$ $\Der(A)/\IDer(A)$. 
See \cite[Chapter 9]{Weibel} for more details on Hochschild cohomology.
If $A$ is commutative, then $HH^1(A)\cong$ $\Der(A)$. A $k$-algebra $A$
is {\it symmetric} if $A$ is isomorphic to its $k$-dual $A^*$ as an
$A$-$A$-bimodule; this definition implies that $A$ is 
finite-dimensional.

\begin{Theorem}[{\cite[Theorem 3.1]{BKL}}] \label{soclederivation}
Let $A$ be a symmetric $k$-algebra and let $E$ be a maximal semisimple 
subalgebra. Let $f \colon A\to A$ be an $E$-$E$-bimodule homomorphism 
satisfying $E+J(A)^2\subseteq$ $\ker(f)$ and $\Im(f)\subseteq$ 
$\soc(A)$. Then $f$ is a derivation on $A$ in $\soc_{Z(A)}(\Der(A))$, 
and if $f\neq$ $0$, then $f$ is an outer derivation of $A$. In 
particular, we have
$$\sum_S\ \dim_k(\Ext^1_A(S,S)) \leq \dim_k(\soc_{Z(A)}(\HH^1(A)))$$
where in the sum $S$ runs over a set of representatives of the 
isomorphism classes of simple $A$-modules.
\end{Theorem}

\begin{Corollary}[{\cite[Corollary 3.2]{BKL}}] 
\label{soclederivationlocal}
Let $A$ be a local symmetric $k$-algebra. Let $f \colon A\to A$ 
be a $k$-linear map satisfying $1+J(A)^2\subseteq$ $\ker(f)$ and 
$\Im(f)\subseteq$ $\soc(A)$. Then $f$ is a derivation on $A$ in 
$\soc_{Z(A)}(\Der(A))$, and if $f\neq$ $0$, then $f$ is an outer 
derivation of $A$. In particular, we have
$$\dim_k(J(A)/J(A)^2) \leq \dim_k(\soc_{Z(A)}(\HH^1(A)))\ .$$
\end{Corollary}

\begin{Theorem}[{Jacobson \cite[Theorem 1]{Jac43}}] \label{kPLiesimple}
Let $P$ be a finite elementary abelian $p$-group of order at least $3$.
Then $\HH^1(kP)$ is a simple Lie algebra.
\end{Theorem}

The converse to this theorem holds as well.

\begin{Proposition} \label{LiesimpleP}
Let $P$ be a finite abelian $p$-group. If $\HH^1(kP)$ is a simple Lie 
algebra, then $P$ is elementary abelian of order at least $3$.
\end{Proposition}

\begin{proof}
Suppose that $P$ is not elementary abelian; that is, its Frattini
subgroup $Q=\Phi(P)$ is nontrivial. We will show that the set of 
derivations with image contained in $I(kQ)kP=$ 
$\ker(kP\to kP/Q)$ is a nonzero Lie ideal in $\Der(kP)$, where 
$I(kQ)$ is the augmentation ideal of $kQ$. Indeed, every element in 
$Q$ is equal to $x^p$ for some $x\in$ $P$, and hence every element in 
$I(kQ)$ is a linear combination of elements of the form $(x-1)^p$, 
where $x\in$ $P$. Every derivation on $kP$ annihilates all elements of 
this form (using the fact that $k$ has characteristic $p$), and hence 
every derivation on $kP$ preserves $I(kQ)kP$. Thus there is a canonical 
Lie algebra homomorphism $\Der(kP)\to$ $\Der(kP/Q)$, which is easily 
seen to be nonzero, with nonzero kernel, and hence $\HH^1(kP)$ is not 
simple. The result follows.
\end{proof}

\begin{Remark}
Theorem \ref{onesimpleLie} implies that in fact for any finite 
$p$-group $P$ the Lie algebra $HH^1(kP)$ is simple if and only if $P$ 
is elementary abelian of order at least $3$. The special case with
$P$ abelian, as stated in \ref{LiesimpleP}, will be needed in the
proof of \ref{onesimpleLie}.
\end{Remark}

\section{Auxiliary results}

In order to exploit the hypothesis on $\HH^1$ being simple in the
statement of Theorem \ref{onesimpleLie},  we consider
Lie algebra homomorphisms into the $\HH^1$ of subalgebras and quotients.

\begin{Lemma} \label{restrictZA}
Let $A$ be a finite-dimensional $k$-algebra and $f$ a derivation on $A$.
Then $f$ sends $Z(A)$ to $Z(A)$, and the map sending $f$ to the induced
derivation on $Z(A)$ induces a Lie algebra homomorphism $\HH^1(A)\to$
$\HH^1(Z(A))$. 
\end{Lemma}

\begin{proof}
Let $z\in$ $Z(A)$. For any $a\in$ $A$ we have $az=$ $za$, hence
$f(az)=$ $f(a)z+af(z)=f(z)a+zf(a)=$ $f(za)$. Comparing
the two expressions, using $zf(a)=$ $f(a)z$, yields $af(z)=f(z)a$,
and hence $f(z)\in$ $Z(A)$. The result follows.
\end{proof}

\begin{Lemma} \label{AZAkernel}
Let $A$ be a local symmetric $k$-algebra such that $J(Z(A))A\neq$ $J(A)$.
Then the canonical Lie algebra homomorphism $\HH^1(A)\to$ $\HH^1(Z(A))$ 
is not injective. 
\end{Lemma}

\begin{proof}
Since $J(Z(A))A<$ $J(A)$, it follows from Nakayama's lemma that
$J(Z(A))A+J(A)^2<J(A)$. Thus there is a nonzero linear endomorphism  $f$ 
of $A$ which vanishes on $J(Z(A))A + J(A)^2$ and on $k\cdot 1_A$, with 
image contained in $\soc(A)$. In particular, $f$ vanishes on $Z(A)=$ 
$k\cdot 1_A+J(Z(A))$. By \ref{soclederivationlocal}, the map $f$ is an
outer derivation on $A$. Thus the class of $f$ in $HH^1(A)$ is nonzero, 
and its image in $HH^1(Z(A))$ is zero, whence the result.
\end{proof}

\begin{Lemma} \label{ZAkerJA}
Let $A$ be a local symmetric $k$-algebra and let $f$ be a derivation on
$A$ such that $Z(A)\subseteq$ $\ker(f)$. Then $f(J(A))\subseteq$ $J(A)$.
\end{Lemma}

\begin{proof}
Since $A$ is local and symmetric, we have $\soc(A)\subseteq$ $Z(A)$,
and $J(A)$ is the annihilator of $\soc(A)$.
Let $x\in$ $J(A)$ and $y\in$ $\soc(A)$. Then $xy=0$, hence
$0=$ $f(xy)=$ $f(x)y+xf(y)$. Since $y\in$ $\soc(A)\subseteq$ $Z(A)$, it
follows that $f(y)=0$, hence $f(x)y=0$. This shows that $f(x)$ annihilates
$\soc(A)$, and hence that $f(x)\in$ $J(A)$.
\end{proof}

\begin{Lemma} \label{JnLemma}
Let $A$ be a finite-dimensional $k$-algebra and $J$ an ideal in $A$.

\smallskip\noindent (i)
Let $f$ be a derivation on $A$ such that $f(J)\subseteq$ $J$. Then
$f(J^n)\subseteq$ $J^n$ for any positive integer $n$.

\smallskip\noindent (ii)
Let $f$, $g$ be derivations on $A$ and let $m$, $n$ be positive 
integers such that $f(J)\subseteq$ $J^m$ and $g(J)\subseteq$ $J^n$.
Then $[f,g](J)\subseteq$ $J^{m+n-1}$.
\end{Lemma}

\begin{proof}
In order to prove (i), we
argue by induction over $n$. For $n=1$ there is nothing to prove. If
$n>1$, then $f(J^n)\subseteq$ $f(J)J^{n-1} + Jf(J^{n-1})$. Both
terms are in $J^n$, the first by the assumptions, and the second
by the induction hypothesis $f(J^{n-1})\subseteq$ $J^{n-1}$.
Let $y\in$ $J$. Then $[f,g](y)=$ $f(g(y))-g(f(y))$. We have
$g(y)\in$ $J^n$; that is, $g(y)$ is a sum of products of $n$ elements 
in $J$. Applying $f$ to any such product shows that the image is in 
$J^{m+n-1}$. A similar argument applied to $g(f(y))$ implies (ii).
\end{proof}

\begin{Proposition} \label{JAProp}
Let $A$ be a finite-dimensional $k$-algebra. For any positive integer
$m$ denote by $\Der_{(m)}(A)$ the $k$-subspace of derivations
$f$ on $A$ satisfying $f(J(A))\subseteq$ $J(A)^m$.

\smallskip\noindent (i)
For any two positive integers $m$ and $n$ we have
$[\Der_{(m)}(A),\Der_{(n)}(A)]\subseteq$ $\Der_{(m+n-1)}(A)$.

\smallskip\noindent (ii)
The space $\Der_{(1)}(A)$ is a Lie subalgebra of $\Der(A)$.

\smallskip\noindent (iii)
For any positive integer $m$, the space $\Der_{(m)}(A)$ is an ideal in 
$\Der_{(1)}(A)$.

\smallskip\noindent (iv)
The space $\Der_{(2)}(A)$ is a nilpotent Lie subalgebra of $\Der(A)$.
\end{Proposition}

\begin{proof}
Statement (i) follows from \ref{JnLemma} (ii). The statements (ii) and
(iii) are immediate consequences of (i). Statement (iii) follows from (i) 
and the fact that $J(A)$ is nilpotent.
\end{proof} 

\section{Proof of Theorem \ref{onesimpleLie}}

Let $G$ be a finite group and $B$ a block of $kG$. Suppose that $B$
has a single isomorphism class of simple modules. If $B$ is nilpotent 
and $P$ a defect group of $B$, then by \cite{Punil}, $B$ is Morita 
equivalent to $kP$, and hence there is a Lie algebra isomorphism 
$\HH^1(B)\cong$ $\HH^1(kP)$. Thus if $B$ is nilpotent with an
elementary abelian defect group $P$ of order at least $3$, then 
$\HH^1(B)$ is a simple Lie algebra by \ref{kPLiesimple}. 

Suppose conversely that $\HH^1(B)$ is a simple Lie algebra. If $J(B)=$ 
$J(Z(B))B$, then $B$ is nilpotent with an abelian defect group $P$ by
\ref{abeliandefectnilpotent}. As before, we have $\HH^1(B)\cong$ 
$\HH^1(kP)$, and hence \ref{LiesimpleP} implies that $P$ is elementary 
abelian of order at least $3$. 

Suppose that $J(Z(B))B\neq$ $J(B)$. Let $A$ be a basic algebra of $B$.
Then $J(Z(A))A\neq$ $J(A)$. Moreover, $A$ is local symmetric, since 
$B$ has a single isomorphism class of simple modules. Thus
$\soc(A)$ is the unique minimal ideal of $A$. We have $J(A)^2\neq$ 
$\{0\}$. Indeed, if $J(A)^2=\{0\}$, then $\soc(A)$ contains $J(A)$,
and hence $J(A)$ has dimension 1, implying that $A$ has dimension $2$.
In that case $B$ is a block with defect group of order $2$. But
then $\HH^1(A)\cong$ $\HH^1(kC_2)$ is not simple, a contradiction.
Thus $J(A)^2\neq$ $\{0\}$, and hence $\soc(A)\subseteq$ $J(A)^2$.
By \ref{AZAkernel}, the canonical Lie algebra homomorphism $\HH^1(A)\to$
$\HH^1(Z(A))$ is not injective. Since $\HH^1(A)$ is a simple Lie 
algebra, it follows that this homomorphism is zero. In other words, 
every derivation on $A$ has $Z(A)$ in its kernel. It follows from
\ref{ZAkerJA} that every derivation on $A$ sends $J(A)$ to $J(A)$.
Thus, by \ref{JnLemma}, every derivation on $A$ sends $J(A)^2$ to
$J(A)^2$. This implies that the canonical surjection $A\to$
$A/J(A)^2$ induces a Lie algebra homomorphism $\HH^1(A)\to$
$\HH^1(A/J(A)^2)$. Note that the algebra $A/J(A)^2$ is commutative
since $A$ is local.
Since $J(A)^2$ contains $\soc(A)$, it follows that the kernel of 
the canonical map $\HH^1(A)\to$ $\HH^1(A/J(A)^2)$ contains the
classes of all derivations with image in $\soc(A)$. Since there
are outer derivations with this property, it follows from the
simplicity of $\HH^1(A)$ that the 
canonical map $\HH^1(A)\to$ $\HH^1(A/J(A)^2)$ is zero. Using
that $A/J(A)^2$ is commutative, this implies that every derivation
on $A$ has image in $J(A)^2$. But then \ref{JAProp} implies that
$\Der(A)=$ $\Der_{(2)}(A)$ is a nilpotent Lie algebra. Thus $\HH^1(A)$
is nilpotent, contradicting the simplicity of $\HH^1(A)$. 
The proof of Theorem \ref{onesimpleLie} is complete.

\begin{proof}[{Proof of Theorem \ref{onetrivialLie}}]
Denote by $A$ a basic algebra of $B$. Since $B$ has a unique 
isomorphism class of simple modules and a nontrivial defect group, it 
follows that $A$ is a local symmetric algebra of dimension at least 
$2$. By \ref{soclederivationlocal} we have $\dim_k(\HH^1(A)))\geq$
$\dim_k(J(A)/J(A)^2)$. Thus $\dim_k(\HH^1(A))\geq$ $1$.
Moreover, if $\dim_k(\HH^1(A))=1$, then $\dim_k(J(A)/J(A)^2)=1$, and
hence $A$ is a uniserial algebra. In that case $B$ is a block with a 
cyclic defect group $P$ and a unique isomorphism class of simple 
modules, and hence $B$ is a nilpotent block.
Thus $A\cong$ $kP$. We have $\dim_k(\HH^1(kP))=$ $|P|$, a contradiction.
The result follows.
\end{proof}

\begin{Remark}
All finite-dimensional algebras in this paper are split thanks to the 
assumption that $k$ is algebraically closed. It is not hard to see that 
one could replace this by an assumption requiring $k$ to be a splitting 
field for the relevant algebras. The statements \ref{restrictZA} and 
\ref{JnLemma} do not require any hypothesis on $k$.
\end{Remark}

\bigskip\bigskip

\noindent\smallskip
{\it Acknowledgement.} The present paper was partially funded by
EPSRC grant EP/M02525X/1 of the first author.


\end{document}